\documentclass[11pt,reqno]{amsart}
\usepackage{geometry} 
\usepackage{hyperref}
\usepackage{bbm,bm}
\usepackage{xcolor}
\usepackage{comment}
\usepackage{float}
\usepackage{enumerate}   

\newtheorem{prop}{Proposition}
\newtheorem{thm}{Theorem}
\newtheorem{lemma}{Lemma}
\newtheorem{assume}{Assumption}
\newtheorem{remark}{Remark}
\newtheorem{definition}{Definition}
\newtheorem{cor}{Corollary}

\numberwithin{prop}{section}
\numberwithin{thm}{section}
\numberwithin{lemma}{section}
\numberwithin{assume}{section}
\numberwithin{remark}{section}
\numberwithin{equation}{section}
\numberwithin{definition}{section}
\numberwithin{cor}{section}
\numberwithin{Ex}{section}

\def \E {\mathbb{E}}
\def \P{\mathbb{P}}
\def \R {\mathbb{R}}
\def \P {\mathbb{P}}

\def \cF {\mathcal{F}}
\def \cL {\mathcal{L}}

\def \0 { {\bf{0}}}

\def \bs{\boldsymbol}

\usepackage{graphicx}

\makeatletter
\@namedef{subjclassname@2020}{%
  \textup{2020} Mathematics Subject Classification}
\makeatother

\title{Walsh spider diffusions as time changed multi-parameter processes}\thanks{E. Bayraktar is partially supported by the National Science Foundation under grant DMS-2106556 and by the
Susan M. Smith chair.  J. Zhang is supported by the National Natural Science Foundation of China under Grant No.12201113. }

\author[]{Erhan Bayraktar} \address{Department of Mathematics, University of Michigan}
\email{erhan@umich.edu}
\author[]{Jingjie Zhang}
\address{China School of Banking and Finance, University of International Business and Economics}
\email{jingjie.zhang@uibe.edu.cn}
\author[]{Xin Zhang} 
\address{Department of Finance and Risk Engineering, New York University}
\email{xz1662@nyu.edu}

\begin{document}

\begin{abstract}
Inspired by allocation strategies in multi-armed bandit model,  we propose a pathwise construction of Walsh spider diffusions.  For any infinitesimal generator on a star shaped graph,  there exists a unique time change associated with a multi-parameter process such that the time change of this multi-parameter process is the desired diffusion. The time change has an interpretation of time allocation of the process on each edge,  and it can be derived explicitly from a family of equations.

\end{abstract}

\maketitle

\section{Introduction}

A Walsh spider diffusion is a singular process on a star shaped graph which behaves like a one-dimensional diffusion on the interior of each edge. Once it hits the vertex, it is kicked away from the vertex in a reflecting manner with a random direction according to some fixed distribution. Indicated by its name, such diffusions were first constructed explicitly by Walsh in \cite{AST_1978__52-53__37_0} as an example of one dimensional diffusions with discontinuous local times. Since then, there has been a continuing interest in both theory and applications of such processes. Extensions of the original process, which was defined on $2$ rays, have been developed for $3$ rays \cite{frank1984random} and $n$ rays \cite{varopoulos1985long}.  Several constructions of Walsh diffusions have been proposed in the literature, including those based on infinitesimal generator \cite{baxter1984equivalence}, Feller's semigroup theory \cite{MR1022917}, resolvents \cite{MR701528}, martingale problem \cite{MaOh23}, and excursion theory \cite{MR859838}. An analogue of It\^{o}'s formula has been established in \cite{KARATZAS20191921}, and the ergodic behavior of Walsh diffusions has been studied in \cite{MR4003554}. Stochastic control and optimization problems on this star-shaped model have been analyzed by \cite{BayraktarErhan2021EoWB,MR3795064,KARATZAS20191921}. Moreover, the connection between queueing theory and Walsh Brownian motion has been justified in \cite{AtCo19}.



In this paper, we propose a new construction of Walsh diffusions as a time change of multi-parameter processes; see Theorem~\ref{thm}. 
An $N$-parameter process can be viewed a collection of random variables indexed by $\R^N_+$. The theory of multi-parameter processes offers an elegant extension to the existing one-parameter theory, and finds diverse applications in various fields, including functional analysis, analytic number theory, sequential stochastic optimization, and data science, and we refer readers to \cite{CairoliR2011Sso,10.3389/fphy.2021.641859,khoshnevisan2006multiparameter} for a nice introduction.

For any infinitesimal generator on a star shaped graph with $N$ edges, we assign each edge a proper reflecting diffusion process $Y_i$ and thus obtain a  multi-parameter process $\mathbb R_+^N \ni (t_1,\dotso,t_N) \mapsto (Y_1(t_1),\dotso, Y_N(t_N))$. 
Inspired by the allocation strategies in 
multi-armed bandit problem \cite{10.1214/aoap/1028903380,MR905347}, we design a random time change $T(t)=(T_1(t),\dotso, T_N(t))$ with respect to the multi-parameter filtration,  which can be interpreted as time allocation of the process on each edge and is determined uniquely by a family of explicit equations involving local times. The construction of the random time change $T(t)$ ensures that the occupation time at the vertex in each direction is distributed according to a given proportion. Furthermore, at any given time $t$, exactly one of $\{T_1(t), \cdots, T_N(t)\}$ satisfies $dT_i(t) = dt$, while $dT_j(t) = 0$ for all $j \ne i$.    After establishing the time-homogeneity of the random time change and the strong Markov property of the multi-pamamerater process, we demonstrate that the time changed multi-parameter process $(Y_1(T_1(t)),\dotso, Y_N(T_N(t)))$ is a diffusion characterized by the desired infinitesimal generator. Additionally, our construction enables us to readily obtain the long-term average occupation time of each edge, as shown in Corollary~\ref{ergodic}.

 The well-known skew Brownian motion is a simple example of the Walsh spider diffusion, where the graph has $2$ edges and the diffusion behaves like a reflecting Brownian motion in the interior of each edge. Different constructions and generalized versions of skew Brownian motion have been studied since the concept was first mentioned by  Ito and McKean \cite[Section 4.2, Problem 1]{ito2012diffusion}; see \cite{10.1214/154957807000000013} and references therein. Among them, ``follow the leader" construction of variably skew Brownian motion 
 in \cite{10.1214/ECP.v5-1018} is closely related to our work. It provided a weak solution to a stochastic differential equation involving local time as a time change of a two-parameter process. While the time change therein is determined by comparing running maximum of two Brownian motions, we derive the time change explicitly through a family of equations of local times. This family of equations represents two conditions imposed on the time change straightforwardly: \emph{(i)} the sum of time allocations on each edge equals the total running time; \emph{(ii)} the local time of each edge at vertex is proportional to the probability of choosing each edge at any time. Our approach not only provides clear intuition, but also enables us to generalize the results from $N = 2$ and reflecting Brownian motion to $N\geq 3$ and general reflecting diffusion. Furthermore, we adopt a different argument in the proof of our main theorem, which is mainly based on time homogeneity of the time change (see Proposition~\ref{prop:T}(iii)) and the strong Markov property of multi-parameter process.

In the next section, we introduce some notions of multi-parameter process and our main result. In Section 3, we present an explicit construction of Walsh diffusion and prove our main result.  Generalizing our approach to  diffusions on a metric graph is also discussed.

\section{Preliminary and main result}\label{sec2}

 Let $\Gamma$ be a star shaped metric graph with one internal vertex $v$ and edges $e_1, \dotso, e_N$. A  Walsh spider diffusion on $\Gamma$ is characterized by an infinitesimal operator $A$ with domain $\mathcal{D}(A)$. Let us introduce the following notations to define $(A, \mathcal{D}(A))$.

For each $i=1,\dotso, N$, a coordinate function $C_i$  is chosen on $e_i$ such that $e_i$ is homeomorphic to an interval $I_i$ of real line and $C_i(v)=0$.  $I_i$ can be either $[0,\infty)$ or $[0,l_i]$ where $l_i>0$ representing the length of edge $e_i$. 
Let us denote by $\partial\Gamma$ the boundary of $\Gamma$, i.e., $\partial\Gamma=\{x: \, x \in e_i, \, C_i(x)=l_i < \infty, i=1,\dotso, N \}$. We equip the star shaped graph $\Gamma$ with the tree metric $d$
\begin{align*}
d(x,x')= 
\begin{cases}
|C_i(x)-C_i(v)|+|C_j(x')-C_j(v)|, \, \quad & i \not= j, \, x \in e_i, \, x' \in e_j, \\
|C_i(x)-C_i(x')|,  & x,x' \in e_i.
\end{cases}
\end{align*}
  A function $f:\Gamma \to \R$ is said to be continuous if $f$ is continuous with respect to the tree metric $d$. Denote by $\mathcal{C}_0(\Gamma)$ the space of continuous functions on $\Gamma$ that vanish at infinity.  We say a function $f \in \mathcal{C}_0(\Gamma)$ belongs to $\mathcal{C}_0^{\infty}(\Gamma)$ if $f|_{e_i^o}$ is smooth, and all derivatives can be naturally extended to $e_i$ for all $i=1,\dotso, N$, where $e_i^o$ denotes the interior of $e_i$. 
 
 For a function $f:\Gamma \to \R$, we denote its derivative at vertex $v$ in the direction of edge $i$ by
 \begin{align*}
 	D_if(v):= \lim\limits_{x \in e_i^o, x\to v} \frac{f(x)-f(v)}{|C_i(x)-C_i(v)|},
 \end{align*}
 if the limit exists, and we denote its derivative at $y \in e_i$, $y \ne v$  by
 \begin{align*}
     Df(y):= \lim\limits_{x \in e_i^o, x\to y} \frac{f(x)-f(y)}{|C_i(x)-C_i(y)|},
 \end{align*}
if the  limit exists.

 Let $\{b_i, \sigma_i\}_{i=1,\dotso, N}$ be bounded and Lipschitz functions on $I_i$, and $\{\sigma^2_i\}_{i=1,\dotso,N}$ are uniformly strictly positive. Define the operators $\cL_i$  via 
 \begin{align*}
 \cL_i f(x)= \frac{1}{2} \sigma_i^2(y_i) \frac{d^2f}{d{y_i}^2}(y_i)+ b_i(y_i) \frac{df}{dy_i}(y_i), \quad x \in e^o_i, \, \,  C_i(x)=y_i, \, \, i=1,\dotso, N. 
 \end{align*}
Let $\{\alpha_{i}:\, i = 1, \cdots,  N\}$ be some fixed positive constants such that $\sum_{i=1}^N \alpha_{i} = 1$. Define an operator $A$ with domain $\mathcal{D}(A) \subset \mathcal{C}_{0}^{\infty}(\Gamma)$ as the following.
\begin{align}\label{eq:infinitesimal} 
& Af(x):=\cL_if(x), \, x \in e_i, \, i=1,\dotso, N, \notag \\
& \mathcal{D}(A):= \left\{f \in \mathcal{C}_{0}^{\infty}(\Gamma): \, Af \in \mathcal{C}_0(\Gamma), \,  \sum_{j=1}^N \alpha_{j} D_j f(v)=0, \, Df(y)=0, \, \forall y \in \partial\Gamma \right\}.  
\end{align}
Note that the condition $Af \in \mathcal{C}_0(\Gamma)$ implies the infinitesimal changes of $f$ coincide at the vertex $v$ in every direction, i.e., 
$\cL_if(v)=\cL_j f(v)$ for any $1 \leq i,j \leq N$. 

It was proved in \cite{MR1022917,MR1245308} using Hille-Yosida theorem that
there exists a continuous Feller process $X(\cdot)$ on $\Gamma$ with infinitesimal generator $(A,\mathcal{D}(A))$, i.e., 
\begin{align*}
\lim\limits_{ t \to 0} \frac{P_t f(x)-f(x)}{t} =A f(x), \quad \text{for $x \in \Gamma$ and $f \in \mathcal{D}(A)$},
\end{align*}
where $P_tf(x):= \E[f(X(t)) \, | \, X_0=x]$ is the semigroup generated by $X(\cdot)$. In words, on the interior of each edge, $e_i^o$, the process $X(\cdot)$ behaves like a diffusion with drift $b_i$ and volatility $\sigma_i$. And once hitting the vertex $v$, the process will be assigned independently a direction $e_j$ with probability $\alpha_{j}>0$; see e.g. \cite{MR1022917,KARATZAS20191921}. 

\begin{remark}
    The domain $\mathcal{D}(A)$ determines the behavior of process at the vertex. In the simplest case $N=1$, $\alpha_1=1$, $b_1\equiv 0$, $\sigma_1 \equiv 1$, with $\mathcal{D}(A)$ defined in \eqref{eq:infinitesimal} we get a reflecting Brownian motion on the interval $I_1$. Another choice of domain $\tilde{\mathcal{D}}(A):=\left\{ f\in \mathcal{C}_{0}^{\infty}(\Gamma): \, Af \in \mathcal{C}_0(\Gamma), \,  f(v)=0  \right\}$ provides a Brownian motion with absorption at $v$; see e.g. \cite{ItMc63}.
    
Suppose there exists a Walsh diffusion $(X_t)$ with the prescribed behavior at the vertex $v$. It is not hard to verify that to make the limit $\lim_{t \to 0}\frac{P_tf(v)-f(v)}{t}$exist, it is necessary to require $f \in \mathcal{D}(A)$; see the proof of Proposition~\ref{prop:verification}. In words, the weighted average derivative of $f$ in each direction should be $0$ so that the infinitesimal generator is well-defined for such reflecting diffusion.
\end{remark}

The purpose of this paper is to provide a pathwise construction of Walsh diffusion, which is achieved by showing that a Walsh diffusion can be represented as a time change of a multi-parameter process. Let us first introduce some notations and terminology about multi-parameter process. All filtrations of stochastic processes are assumed to satisfy the usual conditions in this paper. 

Throughout,  the $i$-th coordinate of any point $\bf s \in \R^N_+$ is written as $s_i$.  The order structure of $\R^N_+$ is explained as the following.
For any $\bs t,  \bs s \in \R^N_+$,  we write $\bs s \leq \bs t \,(\bs s < \bs t)$ or equivalently $\bs t \geq \bs s \,(\bs t > \bs s)$  if for all $i = 1, \dotso, N$,  $s_i \leq t_i \,(s_i <  t_i)$.  

On a probability space $(\Omega, \mathcal{F}, \mathbb{P})$, a collection of sub $\sigma$-algebras, $\mathbb{F}=(\mathcal{F}_{\bs t},  \bs t \in \R^N_+)$, is an $N$-parameter filtration if  $\bs s \leq \bs t $ implies that $\mathcal{F}_{\bs s} \subset \mathcal{F}_{\bs t}$.  An $N$-parameter stochastic process $Z = (Z_{\bs t},  \bs t \in \R^N_+)$ is adapted to the filtration $\mathbb{F}$ if for all $\bs t \in \R^N_+$,  $Z_{\bs t}$ is $\mathcal{F}_{\bs t}$-measurable.



\begin{definition} 
A random variable $\bs \tau = (\tau_1, \cdots, \tau_N) \in \mathbb{R}^N_+$ is a stopping point w.r.t. the filtration $\mathbb{F}$, if the following holds for any $\bs s \in \R^N_+$
\[
\bigcap_{i = 1}^{N} \{\tau_i \le s_i\} \in \cF_{\bs s}.
\]
For any stopping point $\bs \tau$, we associate it with a $\sigma$-algebra 
\begin{align*}
\cF_{\bs \tau}:= \left\{A \in \cF: \, A \cap \{\bs \tau \leq \bs s \} \in \cF_{\bs s} \text{ for any $\bs s \in \R_+^N$} \right\} .
\end{align*}
\end{definition}

The concept of stopping point is a generalization of stoping time in the context of multi-parameter processes. Similarly, we define multi-parameter random time change as a family of stopping points that satisfies certain conditions.

\begin{definition}
A $\mathbb{R}^{N}_+$-valued process $\bs T(\cdot) = (T_1(\cdot), \cdots, T_N(\cdot))$ is called a multi-parameter random time change  w.r.t. the filtration $\mathbb{F}$ if  the following two properties hold:
\begin{enumerate}[(i)]
\item $\bs T$ has continuous non-decreasing sample paths on $\R_+^N$, i.e.,  $t \mapsto T_i(t)$ is non-decreasing and continuous for each $i=1,\dotso, N$ almost surely; 

\item  $\bs T(t) $ is a stopping point w.r.t. the filtration $\mathbb{F}$ for each  $t\ge 0$. 
\end{enumerate}
\end{definition}

We are ready to state our main theorem, which says that a Walsh diffusion can be represented as a time change of a multi-parameter process. Furthermore, the time change is the unique solution to a family of equations. 

To this end, we identify each point $x \in \Gamma$ as a vector $\bs y=(y_1,\dotso, y_N)$ in $\R^N$,  where for each $i=1,\dotso, N$, $y_i=C_i(x)$ if $x\in e_i$ and $y_i=0$ if $x \not \in e_i$.
It is clear that the map 
\begin{equation}\label{eq:representation}
\bs y(x):= (y_1,\dotso,y_N) 
\end{equation}
is a continuous injection from $\Gamma$ to $\R^N$. 


In the rest of the paper, we fix a starting point $ x^0 \in \Gamma$ with $\bs y(x^0)=(y_1^0,\dotso,y_N^0)$ and an infinitesimal generator $(A, \mathcal{D}(A))$ as in (\ref{eq:infinitesimal}). For each $i=1,\dotso,N$, there exists a reflecting diffusion process $Y_i(\cdot)$ on a filtered probability space $(\Omega^i, \cF^i, (\cF^i_t), \P^i  )$ such that $Y_i(\cdot)$ has infinitesimal generator $\mathcal{L}_i$, reflects at end point(s) on interval $I_i$, and $Y_i(0)=y_i^0$; see e.g. \cite{Pi24} for constructions of reflecting diffusions. Throughout the paper, it is convenient to take $\Omega^i=\mathcal{C}([0,\infty);\mathbb R)$ and $Y_i(\cdot)$ to be the canonical process on $\Omega^i$.   Denote by $L^{i}(\cdot)$ the local time of semi-martingale $Y_i$ at the vertex $v$, where we adopt the definition of local time from \cite[Chapter VI]{MR1725357}. To guarantee the strong existence of $Y_i$ and the recurrence of $Y_i$ at $0$, we make the following assumption. Thanks to \cite[Proposition 5.5.22]{KaSh91}, each $Y_i$ is recurrent at $0$ under Assumption~\ref{assume} \emph{(ii)}.
\begin{assume}\label{assume}
 $ $
    \begin{enumerate}
        \item[(i)] The coefficients  $\{b_i,\sigma_i\}_{i=1,\dotso,N}$ are bounded and Lipschitz continuous, and $\{\sigma_i^2\}_{i=1,\dotso,N}$ are uniformly strictly positive. 
        \item[(ii)] For each $i=1,\dotso,N$, we have either $|I_i|=l_i<\infty$, or $$\int_0^{\infty} \exp\left({-2\int_0^{\xi} \frac{b_i(h)}{\sigma_i^2(h)}\,dh}\right)\,d\xi=\infty.$$
    \end{enumerate}
\end{assume}

We take the $N$-parameter process $(\bs Y(\bs s)=(Y_1(s_1), \dotso, Y_N(s_N)), \bs s \in \mathbb{R}_+^N)$ on the filtered probability space $(\Omega, \cF, (\cF_{\bs s}), \P)$ where 
\begin{align}\label{eq:multiprobspace}
 \Omega&:= \Omega^1 \times \dotso \times \Omega^N, \quad \quad \quad \quad \quad  \quad 
 \cF := \sigma^{\mathbb P} \left(\cF^1 \otimes \dotso \otimes \cF^N\right), \\
 \P&:=\P^1 \otimes \dotso \otimes \P^N, \quad \quad \quad \quad \quad  \quad 
 \mathcal{F}_{\bs s}:= \sigma^{\mathbb P} \left(\mathcal{F}^1_{s_1} \otimes \dotso \otimes \mathcal{F}^N_{s_N}\right), \, \, \forall \bs s \in \R_+^N. \notag
\end{align}
We adopt the convention that $\sigma^{\mathbb P}(\mathcal{G})$ represents the complete $\sigma$-algebra generated by $\mathcal{G}$ under $\mathbb P$ for any $\mathcal G \subset \cF^1 \otimes \dotso \otimes \cF^N$. In Lemma~\ref{lem:rightcontinuity}, we will verify the right-continuity of the $N$-parameter filtration $(\cF_{\bs s}, \bs s \in \mathbb{R}_+^N)$, i.e., $\mathcal{F}_{\bs t}=\bigcap_{\bs s > \bs t} \cF_{\bs s}$ for any $\bs t \in \mathbb{R}_+^N$, given that each filtration $\mathcal F^{i} = (\mathcal F^{i}_{s_i}, s_i \in \mathbb{R}_+)$ satisfies the usual conditions for $i= 1,2, \cdots, N$.

\begin{thm}\label{thm}
Under Assumption~\ref{assume}, given  the $N$-parameter process $\bs Y(\cdot)$ there exists a unique random time change $\bs T(\cdot)$  w.r.t. the filtration $(\mathcal{F}_{\bs s})$ such that  $$ X(\cdot):=(Y_1(T_1(\cdot)), \dotso, Y_N(T_N(\cdot)))$$
is a diffusion on $\Gamma$ with filtration $(\mathcal{F}_{\bs T(t)})$ that starts from position $x^0$ with infinitesimal generator $(A, \mathcal{D}(A))$.

Furthermore, almost surely for each $t \geq 0$, $(T_i(t)=s_i)_{i = 1}^{N}$ is the unique solution of equations
\begin{equation*}
\sum_{i=1}^{N} s_i = t, \, \, \, \text{and} \, \, \, \frac{L^{i}(s_i)}{\alpha_{i}}=\frac{L^{j}(s_j)}{\alpha_{j}} \, \, \text{for all} \, \, 1 \leq i,j \leq N. 
\end{equation*}     
\end{thm}

\begin{remark}
It was proved in \cite{MR1487755} that Walsh Brownian motion does not generate a Brownian filtration for $N \geq 3$. To this end, it showed that the property of being ``cozy'' for filtered probability spaces is invariant under morphism. Then the claim follows from the fact that any filtered probability space generated by Brownian motion is ``cozy'', while Walsh Brownian motion is not. A relatively straightforward proof can been found in \cite[Chapter 1.3]{yan2018topics} 

 The multi-parameter filtration \eqref{eq:multiprobspace} considered in this paper differs fundamentally  from one-parameter filtrations, as the time index takes values in $\mathbb R_+^N$ rather than $\mathbb R_+$.
Specifically, let $(\mathcal F^i_t)$, $i=1,\dotso,N$, be chosen as Brownian filtrations. According to Theorem {\ref{thm}}, the Walsh Brownian motion $X$ is not adapted to the Brownian filtration $(\otimes_{i=1}^{N} \mathcal F^i_t)$. Instead, $X$ is adapted to $(\mathcal{F}_{\bs T(t)})$, which is a time change of multi-parameter filtration $(\mathcal F_{\bs s})$ in (\ref{eq:multiprobspace}). Actually, the fact that the filtration $(\mathcal F_{\bs T(t)})$ is not a Brownian filtration comes fundamentally from the lack of total order in $R_+^N$. Because of the partial ordering of $R_+^N$, the filtration $(\mathcal F_{\bs T(t)})$ takes into account the ``order direction" given by $\bs T(t)$, for which the information increase
cannot be determined with a Brownian filtration. 

\end{remark}

\section{Proof of Main Result} 

We start with the construction of a random time change $\bs T$ with respect to the multi-parameter filtration $(\mathcal{F}_{\bs s})$. Define the left-continuous inverse of $L^{i}(\cdot)/\alpha_i$ as  
\begin{equation} \label{eq:hi}
h_i(l) := \inf\left\{t\ge 0: \frac{1}{\alpha_i} L^{i}(t) \ge l \right\},  \, \,  l \in \mathbb{R}_+,
\end{equation}
for $i=1, \dotso, N$, and denote their sum by \begin{equation} \label{eq:h}
h(l) := \sum_{i = 1}^{N}  h_i(l).    
\end{equation}
It is clear that $h$ and $h_i,\, i=1,\dotso, N$, are strictly increasing functions and transform local time to real time. Let us also take the left continuous inverse of $h$, which transform real time to local time.
\begin{align}\label{eq:Xlocaltime}
L(t) := \inf\{l \geq 0: h(l) \geq t\}.
\end{align}
Since $h(l)$ is a strictly increasing function, its left-continuous inverse $L(t)$ is actually continuous. We will see that $L(t)$ is the local time of process $X$ at the vertex $v$ in \eqref{eq:localtimesveri}.

 We now introduce the random time change $\bs T(t) = (T_1(t), T_2(t), \cdots, T_N(t))$, where each $T_i(t)$ has the interpretation of the total time that the Walsh diffusion process spends on the edge $i$ up to time $t$. The main idea of our construction is the following: \emph{(i)} impose the condition \[\frac{L^1(T_1(t))}{\alpha_1}=\dotso=\frac{L^N(T_N(t))}{\alpha_N},\] so that at each time $t$ the occupation time of vertex $v$ for each $Y_i$ is proportional to $\alpha_i$; \emph{(ii)} at each time $t$ , only the process $Y_i$ for which $\Delta h_i(L(t)) := h_i(L(t)+) - h_i(L(t))>0$ will evolve, i.e. $dT_i(t)=dt$, while all other processes $Y_j$ with $\Delta h_j(L(t))  = 0$ remain unchanged, i.e. $dT_j(t)=0$. 
 Precisely, we have the following definition
 \begin{definition}
     Given the reflecting diffusion processes $(Y_i(\cdot))_{i = 1}^N$ and the filtered probability space $(\Omega, \mathcal{F}, (\mathcal{F}_{\bs s}),\mathbb{P})$ in (\ref{eq:multiprobspace}), we define $\bs T = (T_1(\cdot), \cdots, T_N(\cdot))$ via
      \begin{equation} \label{T:construction}
T_i(t) := h_i(L(t)) +\mathbbm{1}_{\{\Delta h_i(L(t))>0\}}(t-h(L(t))), \quad \forall \, i \in \{1,2,\cdots, N\},
\end{equation}
where $h_i(\cdot), h(\cdot)$ and $L(\cdot)$ are given by (\ref{eq:hi}), (\ref{eq:h}) and (\ref{eq:Xlocaltime}), respectively.
 \end{definition}

\begin{remark}
Note that here the time change $\bs T$ is expressed in terms of the left-continuous inverse $h_i$ of local times, where $h_i(l)$ stands for the starting time of excursion at the local time level $l$. Then by \eqref{T:construction}, it is clear that $t \mapsto T_i(t)$ is increasing. 


\end{remark}

In the following lemma, we show that almost surely there is at most one process with a jump at $h_i(l)$ for all $l \ge 0$.

  \begin{lemma}\label{lem:simulflat} $\{h_i\}_{i=1}^{N}$ have no simultaneous jumps almost surely. That is
for \[\Omega_0:= \left\{\omega \in \Omega: \sum_{l \geq 0} \Delta h_i(l)(\omega)\Delta h_j(l)(\omega) = 0, \forall i \ne  j   \right\}\] we have $\mathbb{P}[\Omega_0]=1$.  
  \end{lemma}
\begin{proof} 
  Notice that $\Omega_0 =  \Omega_1 \cap \Omega_2$,
 where
 \begin{align*}
\Omega_1 & = \left\{\omega\in \Omega: \sum_{l>0} \Delta h_i(l)(\omega)\Delta h_j(l)(\omega) = 0, \forall i \ne  j \right\};\\
\Omega_2 & = \left\{\omega\in \Omega:  \Delta h_i(0)(\omega)\Delta h_j(0)(\omega) = 0, \forall i \ne  j \right\}.
 \end{align*}

 Let us first prove that $\mathbb{P}(\Omega_1) = 1$ using the excursion theory. For each process $Y_i$, point $0$ is regular, i.e., $\mathbb{P}^0(\tau_0 = 0) = 1$ and recurrent, i.e., $\mathbb{P}^{y_i}(\tau_0 < \infty) = 1, \forall y_i \in I_i$, where $\tau_{0}:=\inf\{t \geq 0: \, Y_i(t)=0 \}$ denotes the hitting time of $0$. We denote the excursion process of $Y_i$ from point $0$  by $(e_i(s))_{s>0}$. The excursion process takes values in the  excursion space \[U^{(i)}_{\delta} = U^{(i)} \cup \delta,\] where  $\delta$ is the function that is identically 0,  $U^{(i)}=\cup_{n = 1}^{\infty} U^{(i)}_n$, and \[
U^{(i)}_n = \left\{\text{continuous path } c: [0,\tau]\to I_i \text{ where } \tau = \inf\{t> 0: c(t)= 0\} \text{ satisfies } \frac{1}{n}< \tau <\infty  \right\} .\]
The excursion process is given by \[
e_i(s) = \begin{cases}
    Y_i(h_i(\frac{s}{\alpha_i})+\cdot ):  [0, \Delta h_i(\frac{s}{\alpha_i})] \to I_i  & \text{if } \Delta h_i(\frac{s}{\alpha_i}) > 0;\\
    \delta & \text{if } \Delta h_i(\frac{s}{\alpha_i}) = 0.
\end{cases}
\]
It is known that $e_i(\cdot)$ is a Poisson point process; see e.g. \cite{Bl92} .

For each $n\geq 1$, \[
N_i^{(n)}(l) := \sum_{s \leq l} \mathbbm{1}_{\{e_i(s) \in U_n^{(i)}\}}, \quad l > 0
\] denotes the number of excursions before local time $l$ whose time duration $\Delta h_i(\frac{s}{\alpha_i})$ is greater than $\frac{1}{n}$. Therefore we have \[
N_i^{(n)}(l)  \leq n h_i \left(\frac{l}{\alpha_i}+\right)  < \infty,
\] 
and $N_i^{(n)}(\cdot)$ is a Poisson process with finite jump rate. For any $i \in \{1,2, \cdots, N\}, n\geq 1$, define another Poisson process $\tilde N_i^{(n)}(l): = N_i^{(n)}(\alpha_i l)$. By the independence of $Y_i$, $\tilde N_i^{(n)}(\cdot), i \in \{1,2, \cdots, N\}$ are also independent. According to \cite[Chapter XII, Proposition 1.5]{MR1725357}, independent Poisson processes have no simultaneous jumps, i.e.,
\[\mathbb{P}\left(\sum_{l>0} \Delta \tilde N_i^{(n)}(l)\tilde N_j^{(n)}(l) = 0, \forall i \ne j \right) = 1, \quad \forall n\geq 1.\]

Notice that $\Delta h_i(l)(\omega) > 0$ if and only if $L^i(\cdot)$ has a constant stretch at level $\alpha_i l$, i.e., $e_i(\alpha_i l) \ne \delta$, which is equivalent to that there exists $n \geq 1$ such that $\Delta \tilde N_i^{(n)}(l) > 0$. Thus 
\begin{align*}
\mathbb{P}\left(\Omega_1\right)= &\mathbb{P}\left(\left\{\sum_{l>0} \Delta h_i(l)\Delta h_j(l) = 0, \forall i \ne  j \right\}\right) \\
= & \mathbb{P}\left(\bigcap_{n\geq 1} \left\{ \sum_{l>0} \Delta \tilde N_i^{(n)}(l)\tilde N_j^{(n)}(l) = 0, \forall i \ne j \right\} \right) \\
= & \lim_{n\to \infty} \mathbb{P}\left(\sum_{l>0} \Delta \tilde N_i^{(n)}(l)\tilde N_j^{(n)}(l) = 0, \forall i \ne j \right) \\
= & 1.
\end{align*}

Next consider $\Omega_2 = \left\{\omega\in \Omega: \Delta h_i(0)(\omega)\Delta h_j(0)(\omega), \forall i \ne  j \right\}$. For  each $Y_i$, when $Y_i(0) = 0$, we have $\Delta h_i(0) = 0$. If this does not hold, there exists $\epsilon > 0$ such that for all $s> 0$, $h_i(s) > \epsilon$, which implies $L_i(\epsilon) < \alpha_i s$ for any $s>0$. However it contradicts with the fact that $\mathbb{P}[ L^i(\epsilon) > 0 \,|\, Y_i(0) = 0] = 1$. Therefore in both cases of the initial position $x = v$ and $x \ne v$, there is at most one $i \in \{1,2, \cdots, N\}$ such that $\Delta h_i(0)(\omega) > 0$. This leads to $\mathbb{P}(\Omega_2) = 1$ and hence \[
\mathbb{P}(\Omega_0) = 1 - \mathbb{P}(\Omega_1^c \cup \Omega_2^c) = 1.
\]
 \end{proof}
 
 \begin{remark}
Note that $h_i$ has a jump at $l$, i.e., $h_i(l) < h_i(l+)$, is equivalent to that the local time $L^i(\cdot)$ is flat at  level $\alpha_i l$, i.e., $ \text{Leb}(\{t:\, L^i(t)=\alpha_i l \})>0$.  Hence the absence of simultaneous jumps in processes $\{h_i\}_{i =1}^ N$ is equivalent to the absence of same level of plateaus in processes $\{L^i(\cdot)/\alpha_i\}_{i =1}^ N$. It is known that the local time $L^i(\cdot)$ for process $Y_i$ only exhibits a plateau when the process is on an excursion away from $v$. Therefore, heuristically, this lemma implies that for any $l \geq 0$, at most one process among $\{Y_i\}_{i =1}^ N$ is on an excursion at local time $\{\alpha_i l\}_{i=1}^N$. 
 \end{remark}

The following lemma shows that if the filtration $(\mathcal{F}^i_{s_i})$ is right-continuous and complete for each $i=1,\dotso, N$, then the multi-parameter filtration $(\mathcal{F}_{\bs s})$ defined in \eqref{eq:multiprobspace} is right-continuous as well. 

 \begin{lemma}\label{lem:rightcontinuity}
The filtration $(\cF_{\bs s})$ is right continuous, i.e., $\mathcal{F}_{\bs t}=\bigcap_{\bs s > \bs t} \cF_{\bs s}$ for any $\bs t \in \mathbb{R}_+^N$.
\end{lemma}
\begin{proof}
     Let us take a decreasing sequence $(\bs s^k)_{k \in \mathbb N}$ such that $\bs s^k>\bs t$ and $\lim\limits_{k \to \infty} \bs s^k=\bs t$. Notice that the claim is equivalent to $\mathbb P \left[A \, | \, \cF_{\bs t} \right]= \mathbb P \left[ A \, | \, \bigcap_{k \in \mathbb N} \cF_{\bs s^k}\right]$  for any $A \in \cF$. By a standard monotone class argument, it suffices to prove $\mathbb P \left[A \, | \, \cF_{\bs t} \right]= \mathbb P \left[ A \, | \, \bigcap_{k \in \mathbb N} \cF_{\bs s^k}\right]$ for all $A=A_1\times \dotso \times A_N$ with $A_i \in \cF^i$, $i=1,\dotso, N$.         
     
     Since each filtration $(F^i_{s_i})_{s_i \in \mathbb R_+}$ is complete and right continuous, the probability measure $\mathbb P$ is a product of $\mathbb P^1,\dotso, \mathbb P^N$, we have 
     \begin{align*}
     \mathbb P[ A \, | \, \cF_{\bs t}]&=\prod_{i=1}^N \mathbb P^i[A_i \,| \, \cF^i_{t_i}]=\prod_{i=1}^N \lim\limits_{k \to \infty} \mathbb P^i[A_i \,| \, \cF^i_{s^k_i}]= \lim\limits_{k\to \infty} \prod_{i=1}^N \mathbb P^i [A_i \, | \,\cF^i_{s^k_i} ],
     \end{align*}
     where in the second equality we make use of the backward martingale convergence theorem; see e.g. \cite[Appendix A.2]{GaFr16}. By the same token, the last term is equal to $\lim\limits_{k \to \infty} \mathbb P \left[ A \, | \, \mathcal{F}_{\bs s^k} \right]=\mathbb P \left[ A \, | \, \bigcap_{k \in \mathbb N} \cF_{\bs s^k}\right]$, which completes our proof.
\end{proof}

%

We are now ready to present the key properties of the random time change $\bs T(t)$ defined in (\ref{T:construction}). Before that we recall the definition of shift operator $\bs \theta$ on canonical spaces.

Denote the path space by $\Omega= \mathcal{C}([0,\infty);\R)^{N}$, and each element therein  by $\omega=(\omega_1, \dotso, \omega_N)$
For any non-negative random variable $\bs \tau=(\tau_1, \dotso, \tau_N)=: \Omega \to \R_+^N$, we define the map $\bs\theta_{\bs \tau} : \Omega \to \Omega $ via 
\begin{align*}
\bs \theta_{\bs \tau} ( \omega) (t_1, \dotso, t_N)= ( \omega_1(t_1+ \tau_1(\omega)), \dotso, \omega_N(t_N+\tau_N(\omega))), 
\end{align*}
and its corresponding shift operator $\bs \theta_{\bs \tau}$ via 
\begin{align*}
\bs \theta_{\bs \tau} \xi (\omega):= \xi ( \bs \theta_{\bs \tau} (\omega)) \text{ for any random variable $\xi: \Omega \to \R$}. 
\end{align*}

\begin{prop}\label{prop:T}
The  $\bs T$ defined in \eqref{T:construction} has the following properties. 
\begin{enumerate}[(i)]
\item For every $\omega \in \Omega_0$ with $\Omega_0$ defined in Lemma~\ref{lem:simulflat} and any $t\in \mathbb{R}_+$, $(s_i:=T_i(t))_{i=1}^N$ is the unique solution to 
\begin{equation}\label{eq:familyEq}
\sum_{i=1}^{N} s_i = t, \, \, \, \text{and} \, \, \, \frac{L^{i}(s_i)}{\alpha_{i}}=\frac{L^{j}(s_j)}{\alpha_{j}} \, \, \text{for all} \, \, 1 \leq i,j \leq N. 
\end{equation}
Moreover, we have $L(t)=\frac{L^i(T_i(t)) }{\alpha_i}, \, \forall \, i=1,\dotso, N$, for all $t \geq 0$. 

\item $\bs T$ is a multi-parameter random time change.
\item $\bs T$ is time homogeneous.  That is for any $t,s \geq 0$, we have that 
\begin{align*}
\bs T(s+t)=\bs T(s)+\bs \theta_{\bs T(s)} \bs T(t) \quad a.s. 
\end{align*}
\end{enumerate}
\end{prop}
\begin{proof}
\emph{(i):} First we prove that $\bs T$ defined by \eqref{T:construction} satisfies \eqref{eq:familyEq}. Recall the definition of $T_i$
\begin{align*}
    T_i(t) := h_i(L(t)) +\mathbbm{1}_{\{\Delta h_i(L(t))>0\}}(t-h(L(t))).
\end{align*}
Thanks to Lemma \ref{lem:simulflat}, for every $\omega \in \Omega_0$, $ \Delta h(l)>0$ if and only if there is exactly one $i\in \{1,2, \cdots, N\}$ such that $ \Delta h_i(l)>0$. Thus \[
\sum_{i}  \mathbbm{1}_{\{\Delta h_i(L(t))>0\}} =  \mathbbm{1}_{\{\Delta h(L(t))>0\}}.
\]
It follows from the construction of $\{T_i\}_{i=1}^N$ that \[
\sum_{i}T_i(t) = h(L(t)) + \mathbbm{1}_{\{\Delta h(L(t))>0\}}(t - h(L(t))) = t.
\]
Since $h_i$ is the left-continuous inverse of $L^i/\alpha^i$, it follows directly that when  $\Delta h_i(L(t)) = 0$ 
\[
\frac{L^i(T_i(t)) }{\alpha_i} = L(t).
\]

When $\Delta h_i(L(t)) > 0$, by the definition of left-continuous inverse, $\frac{L^i}{\alpha_i}$ must be constant (on a ``plateau'') at level $L(t)$ over the interval $ \left[h_i(L(t)), h_i(L(t)) +\Delta h_i(L(t)) \right]$. Moreover, by Lemma \ref{lem:simulflat}, $\Delta h_i(L(t))  = \Delta h(L(t)) $, which is no less than $t - h(L(t))$, given that $\Delta h(L(t))  > 0$ and $L$ is the left-continuous inverse of $h$. Thus, \[h_i(L(t))+(t-h(L(t))) \in  \left[h_i(L(t)), h_i(L(t)) +\Delta h_i(L(t)) \right]\] and 
 \[
\frac{L^i(h_i(L(t)+\mathbbm{1}_{\{\Delta h_i(L(t))>0\}}(t-h(L(t)))))}{\alpha_i} = \frac{L^i(h_i(L(t)))}{\alpha_i} = L(t).
\]

Then we prove that for any $s\geq 0$, solution to \eqref{eq:familyEq}  is unique. Suppose there is another family of solutions $\{ s_i'\}_{i=1}^N$ to \eqref{eq:familyEq}. Let us take 
\begin{align*}
c= \frac{L^{1}(s_1)}{\alpha_1}= \dotso= \frac{L^{N}(s_N)}{\alpha_N},
\end{align*}
and 
\begin{align*}
c'= \frac{L^{1}(s'_1)}{\alpha_1}= \dotso=\frac{ L^{N}(s'_N)}{\alpha_N}.
\end{align*}
Since $\sum_{i=1}^N s_i = \sum_{i=1}^N s_i'$, without loss of generality we assume that $s_1 > s'_1, \, s_2 <s'_2$. Therefore by the monotonicity of local times, we have $$c=\frac{L^{1}(s_1)}{\alpha_1} \geq \frac{L^1(s_1')}{\alpha_1}=c'=\frac{L^2(s'_2)}{\alpha_2} \geq \frac{L^2(s_2)}{\alpha_2}=c,$$
and hence $c=c'$. Then $h_1$ and $h_2$ have jumps simultaneously at $l=c$. Indeed, the equalities above imply that $ h_1(c) \leq s_1' < s_1 \leq  h_1(c+)$ and $h_2(c) \leq s_2 < s_2' \leq h_2(c+)$.
It contradicts with Lemma \ref{lem:simulflat}.
\vspace{0.5cm} 

\emph{(ii):} 
According to the definition \eqref{T:construction}, $T_i(t)$ is increasing in $t$ for every $i \in \{1,2,\cdots, N\}$. It remains to show that $T_i(t)$ is continuous in $t$ and $\bs T(t)$ is a stopping point for every $t \in \mathbb R_+$.

By \emph{(i)} we have that for any $0\leq s< t$, $
 \sum_{i} T_i(t) - \sum_{i}T_i(s) = t-s$. Then  \[
 |T_i(t) - T_i(s)| \le |t-s|, \, \,  \forall \, t,s \ge 0,
 \]
 which indicates that $t \mapsto T_i(t)$ is continuous. For continuous and increasing $T_i$, to show that $T$ is a random time change, it suffices to show that $$\{T_1(t) \le s_1,  \cdots,  T_N(t) \le s_N\} \in \mathcal{F}^1_{s_1} \otimes \cdots \otimes \mathcal{F}^N_{s_N},  \quad \forall \, \bs s=(s_1,\dotso,s_N) \in \R_+^N. $$  

To this end, we adopt the idea from \cite[Theorem 3]{MR905347}, showing that $T$ is the limit of stopping points that have exactly only one component increasing at each time $t$ according to a \emph{priority scheme}. 

For a fixed $\epsilon > 0$ and each $i=1,\dotso,N$, define 
\begin{align*}
c^{(i)}_k :=&h_i(\epsilon k)-h_i(\epsilon(k-1)),  \, \, \, \, \forall k \geq 1,
\end{align*}
and a sequence of intervals $\left\{U_{k}^{(i)}: \, k \in \mathbb{N}_+, \, i=1,\dotso, N \right\}$ via 
\begin{align*}
U_{k}^{(i)}: = \left[h(\epsilon(k-1) )+ \sum_{l=1}^{i-1}c^{(l)}_{k},  \,\,h(\epsilon(k-1) ) +  \sum_{l =1}^{i}c^{(l)}_{k}\right).
\end{align*}
Consider a family of stopping points $\bs T^{\epsilon} = (T^{\epsilon}_1, \cdots, T^{\epsilon}_N)$ defined by
\begin{align*}
T^{\epsilon}_i(t) = \sum_{k = 1}^{\infty} \int_0^t \mathbbm{1}_{\{u \in U_{k}^{(i)}\}} \, du.
\end{align*}

Intuitively,  we let $T^{\epsilon}_1$ run with rate $1$ until $\frac{L^{1}(t)}{\alpha_1} $ reaches level $\epsilon$.  Proceed sequentially with each $T^{\epsilon}_{i}$ until $\frac{L^{i}(t)}{\alpha_i} $ reaches level $\epsilon$ for $i = 2, \cdots,  N$.  When  $\frac{L^{N}(t)}{\alpha_N} $ is at level $\epsilon$,  let $T^{\epsilon}_1$ run again until $\frac{L^{1}(t)}{\alpha_1} $ reaches level $2\epsilon$.  Proceed sequentially with each $i = 2, \cdots,  N$ until $\frac{L^{N}(t)}{\alpha_N} $ is at level $2\epsilon$,  and so on.  In the above notation,  $\sum_{j= 1}^{k} c^{(i)}_j = h_i(k\epsilon )$, and for all $i \in \{1,\cdots, N\}$, $\frac{L^{i}(T^{\epsilon}_i(t))}{\alpha_i}  = k\epsilon$ whenever $t = h(\epsilon(k-1) )+ \sum_{l=1}^{N}c^{(l)}_{k}=h(\epsilon k)$.  Indeed, according to the construction $T^{\epsilon}_i(h(\epsilon k))=h_i(\epsilon k)$, and hence $\frac{L^{i}(T^{\epsilon}_i(t))}{\alpha_i}=\frac{L^{i}(h_i(k \epsilon))}{\alpha_i}=k \epsilon$.  We summarize that  $T^{\epsilon}_i(t) = T_i(t) = h_i(k\epsilon)$, $i=1, \dotso,N$, when $t = h(k\epsilon)$. 

With $\epsilon = 2^{-n}$, $\bs T^{\epsilon}(t)$ converges to $\bs T(t)$ as $n \to \infty$ due to the fact that  $t \mapsto T_i(t)$ is increasing and continuous. Fix any $t_0 \in \mathbb R$ and $\delta>0$, we claim that $|T_i^{2^{-n}}(t_0)-T_i(t_0)| < \delta$ for large enough $n$. Indeed, take some $n_0 \in \mathbb N$ and $k\in \mathbb N$ such that $t_0 \in [(k-1)/2^{n_0},k/2^{n_0})$ and $|h_i(k/2^{n_0} ) - h_i((k-1)/2^{n_0})| < \delta/2 $. According to our construction, $T^{2^{-n}}_i((k-1)/2^{n_0})=T_i((k-1)/2^{n_0}) =h_i((k-1)/2^{n_0})$, $T^{2^{-n}}_i(k/2^{n_0})=T_i(k/2^{n_0}) =h_i(k/2^{n_0})$ for any $n \geq n_0$. Therefore since $t \mapsto T_i^{2^{-n}}(t)$ and $ t \mapsto T_i(t)$ are increasing, it must hold true that for any $n \geq n_0$, 
\[|T_i^{2^{-n}}(t_0)-T_i(t_0)| < \max\{|T_i^{2^{-n}}((k-1)/2^{n_0})-T_i(k/2^{n_0})|,|T_i^{2^{-n}}(k/2^{n_0})-T_i((k-1)/2^{n_0})| \} < \delta.\]

In the end, since $\bs T^{\epsilon} (t)$ is a stopping point and the filtration $\mathcal{F}(s)$ is right-continuous by Lemma \ref{lem:rightcontinuity}, its limit $\bs T$ must be a stopping point as well.
  
\vspace{0.5cm}

\emph{(iii)}: 
Take $\Omega_0$ as defined in Lemma~\ref{lem:simulflat}. Let us first prove the result assuming
\begin{align}\label{eq:propTassume}
    \bs \theta_{\bs T(s)}(\omega) \in \Omega_0 \quad \text{ for almost every $\omega \in \Omega_0$.} 
\end{align}
Thanks to \emph{(i)}, it suffices to show that for  almost every $\omega \in \Omega_0$,
\begin{equation}\label{eq:propTnece1}
\sum_{i} T_i(s)(\omega) + \bs \theta_{\bs T(s)} T_i(t)(\omega) = s+t,
\end{equation}
as well as 
\begin{equation}\label{eq:propTnece2}
\frac{L^{i}({T_i(s)+ \bs \theta_{\bs T(s)} T_i(t)})(\omega) }{\alpha_i} \, \, \,  \text{is independent of $i$.}
\end{equation}

According to the definition of $\bs \theta$,  \[
\bs \theta_{\bs T(s)}\bs T(t)(\omega) = \bs T(t)(\bs \theta_{\bs T(s)}(\omega)),
\]
and therefore 
\begin{align*}
    \sum_{i} \bs \theta_{\bs T(s)}T_i(t)(\omega)=\sum_i T_i(t)(\bs \theta_{\bs T(s)}(\omega))=t,
\end{align*}
which justifies \eqref{eq:propTnece1}.

Then we prove \eqref{eq:propTnece2}. Recall a property of local time for continuous semi-martingales \cite[Chapter VI, Corollary 1.9]{MR1725357} 
\begin{align*}
L^{i} (t)(\omega_i)  &= \lim_{\epsilon \to 0} \frac{1}{\epsilon} \int_{0}^{t} \mathbbm{1}_{\{|Y_i(u)(\omega_i)| < \epsilon\}} d [Y_i]_u(\omega_i),
\end{align*}
where $[Y_i]_u$ stands for the quadratic variation of $Y_i$ at time $u$. Hence we obtain that 
\begin{align}\label{eq:localtime}
&L^{i}\left(T_i(s)+ \bs \theta_{\bs T(s)} T_i(t)\right) (\omega) \notag \\
&= \lim_{\epsilon \to 0} \frac{1}{\epsilon} \int_{0}^{T_i(s)(\omega)+ \bs \theta_{\bs T(s)} T_i(t)(\omega)} \mathbbm{1}_{\{|Y_i(u)(\omega)| < \epsilon\}} \, d [Y_i]_u(\omega) \notag \\
&= \lim_{\epsilon \to 0} \frac{1}{\epsilon} \int_{0}^{T_i(s)(\omega)} \mathbbm{1}_{\{|Y_i(u)(\omega)| < \epsilon\}} \, d [Y_i]_u(\omega) \notag \\
& \quad +\lim_{\epsilon \to 0} \frac{1}{\epsilon} \int_{0}^{ T_i(t)(\bs \theta_{\bs T(s)}(\omega))} \mathbbm{1}_{\{|Y_i(u)(\bs \theta_{\bs T(s)}(\omega))| < \epsilon\}} \, d [Y_i]_u(\bs \theta_{\bs T(s)}(\omega))  \notag \\
&= L^{i} \left(T_i(s)(\omega)\right)(\omega)  + L^{i}\left( T_i(t)(\bs \theta_{\bs T(s)}(\omega)) \right) (\bs \theta_{\bs T(s)}(\omega)),
\end{align}
where we use the strong Markov property of $Y_i$. In other words, $L^i$ is a a strong additive functional of $Y_i$; see e.g. \cite[Chapter IV, Section 1]{BlGe68}. Since both $\omega, \bs \theta_{\bs T(s)}(\omega)$ belong to $\Omega_0$, part $\emph{(i)}$ implies that the following quantity is independent of $i$
\begin{align*}
    \frac{L^{i}({T_i(s)+ \bs \theta_{\bs T(s)} T_i(t)})(\omega) }{\alpha_i} =  \frac{L^{i} \left(T_i(s)\right)(\omega) }{\alpha_i} + \frac{L^{i}\left( T_i(t)(\bs \theta_{\bs T(s)}(\omega)) \right) (\bs \theta_{\bs T(s)}(\omega))}{\alpha_i}.
\end{align*}
It completes the proof of \eqref{eq:propTnece2}.

As a final step, we prove the claim \eqref{eq:propTassume}. We show that for any $\bs \theta_{\bs T(s)}(\omega)$ with $\omega \in \Omega_0$, \[\{h_i(l)(\bs \theta_{\bs T(s)}(\omega))\}_{i=1,\dotso,N}\]  have no simultaneous jumps at any local time level $l \geq 0$. Indeed, by a similar argument as \eqref{eq:localtime} we have 
\begin{align}\label{eq:propTnece4}
\frac{L^i(t)(\bs \theta_{\bs T (s)}(\omega))}{\alpha_i}= \frac{L^i(t+T_i(s)(\omega))(\omega)}{\alpha_i}-\frac{L^i(T_i(s)(\omega))(\omega)}{\alpha_i}.
\end{align}
As $\omega \in \Omega_0$, there are no simultaneous jumps of $\{h_i(l)(\omega)\}_{i=1,\dotso,N}$, which can be written as 
\begin{align*}
      \text{Leb}(\{t: L^{i}(t)(\omega) = \alpha_i l\}) \cdot \text{Leb}(\{t: L^{j}(t)(\omega) = \alpha_j l\})= 0, \quad \forall \, l \geq 0, \, \, \forall \, i \not = j.
\end{align*}
Together with \eqref{eq:propTnece4} and the fact that $L^i(T_i(s)(\omega))(\omega)/\alpha_i$ is independent of $i$, one can deduce that 
\begin{align*}
      \text{Leb}(\{t: L^{i}(t)(\bs \theta_{\bs T(s)}(\omega)) = \alpha_i l\}) \cdot \text{Leb}(\{t: L^{j}(t)(\bs \theta_{\bs T(s)}(\omega)\}) = \alpha_j l)= 0, \quad \forall \, l \geq 0, \, \, \forall \, i \not = j,
\end{align*}
which justifies there are no simultaneous jumps of $\{h_i(l)(\bs \theta_{\bs T(s)}(\omega))\}_{i=1,\dotso,N}$ for every $\omega \in \Omega_0$.

\end{proof}
\begin{remark}
Notice that we have a priority scheme when defining stopping points $\bs T^{\epsilon}(t)$ in the proof of Proposition~\ref{prop:T} \emph{(ii)}. However, according to the construction of $\bs T^{\epsilon}$,  $T^{\epsilon}_i(h(k\epsilon)) = h_i(k\epsilon)$ is independent of the priority scheme for every $k \in \mathbb N$. Therefore, as the limit of $\bs T^{\epsilon}(t)$, the stopping point $\bs T(t)$ is also independent of the priority scheme.

\end{remark}

Using the random time change $\bs T = (T_1,  \cdots, T_N)$ constructed above,  we define our diffusion process on $\Gamma$ as the time change of the multi-parameter process, i.e., 
\begin{equation}\label{eq:diffusion}
 X(t) =\left(Y_1(T_1(t)),  \cdots,  Y_N(T_N(t)) \right).
\end{equation}
We have that $X(t) \in \Gamma$ for any $ t \geq 0$ almost surely. Indeed according to Lemma~\ref{lem:simulflat}, at any time $t \geq 0$, at most one $Y_{i}(T_{i}(t))$ is on excursion from $v$, which means at any time $t \geq 0$,  there is at most one $Y_{i}(T_{i}(t))>0$, and $Y_j(T_j(t))=0$ for $j \ne i$.

\begin{remark}\label{rmk:3.4} 
 Heuristically, $T_i(t)$ represents the total time that the Walsh diffusion process spends on edge $i$ up to time $t$ and their sum equals $t$, as shown in Proposition \ref{prop:T}. Moreover (\ref{T:construction}) and Lemma \ref{lem:simulflat} indicate that $(T_i(\cdot), i = 1,2,\cdots, N)$ evolve in one of the following manners. 
 \begin{enumerate}
     \item \emph{Single component increase:} Exactly one of  $(T_i(\cdot), i = 1,2,\cdots, N)$ has a strictly positive rate of change, specifically with rate $1$. This occurs when $\Delta h(L(t))> 0$. For the unique index $i$ such that $\Delta h_i(L(t))>0$, we have $d T_i(t)  = dt$, while $d T_j(t)  = 0$ for all $j \ne i$.
     \item \emph{Simultaneous increase: }All $(T_i(\cdot), i = 1,2,\cdots, N)$ increase simultaneously. This occurs when $\Delta h(L(t)) = 0$, equivalently when $dL(t) > 0$.  As $L(\cdot)$  is the local time of the Walsh diffusion $X$ at vertex $v$ (see Remark \ref{directionProb}), $L(\cdot)$ only increases at $ \{s \geq 0: X(s) = v\}$, which has Lebesgue measure $0$. So we say $(T_i(\cdot), i = 1,2,\cdots, N)$ increase simultaneously at essentially no time.
     
 \end{enumerate}
 
 Accordingly, \eqref{T:construction} can also be written in the differential form $dT_i(t)= \mathbbm{1}_{\{\Delta h_i(L(t))>0\}} dt$ \text{Leb}-$a.s.$ Indeed due to \eqref{eq:hi} and Proposition~\ref{prop:T} (i), we have $l \mapsto h_i(L(l))=T_i(l)$ is Lebesgue-differentiable at $l=t$ whenever  $dL(t)>0$.  Then thanks to the point $(2)$ above, it holds that $dh_i(L(t))=0$ for Leb-a.e. $t$. Moreover, when $\Delta h_i(L(t))>0$, the process $Y_i$ is on an excursion away from the vertex $v$, and hence $L(t)=L(t+\epsilon)$ for small enough $\epsilon >0$. So that we conclude from \eqref{T:construction} that 
 \[
dT_i(t) = dh_i(L(t)) +d\mathbbm{1}_{\{\Delta h_i(L(t))>0\}}(t-h(L(t)))=\mathbbm{1}_{\{\Delta h_i(L(t))>0\}} dt \quad \text{Leb-}a.s.
\]

\end{remark}

Expanding on the concept of strong Markov property of a one-parameter process, we define the strong Markov property of multi-parameter process as the following.

\begin{definition}
 $\bs{Y}(\bs t) = (Y_1(t_1),  \cdots,  Y_N(t_N))$ is a multi-parameter process on a filtered probability space $(\Omega,  \mathcal{F}, (\cF_{\bs s}),\mathbb{P})$ with values in $(\mathbb{R}^N,  \mathcal{B}(\mathbb{R}^N))$.  Let $\bs{\tau} = (\tau_1, \cdots, \tau_N)$ be an almost surely finite stopping point with respect to the filtration $(\mathcal{F}_{\bs s}) $.  We say that the strong Markov property holds at $\bs{\tau}$ for process $\bs{Y}$ if for arbitrary $k$ vectors $\bs{t}_1, \cdots,  \bs{t}_k \in \mathbb{R}^N_+ $ and arbitrary $k$  bounded continuous functions $f_1,\dotso, f_k \in \mathcal{C}_b(\mathbb{R}^N)$,  
\begin{equation}\label{eq:strongmarkov}
\mathbb{E} \left[ \prod_{i=1}^k f_i(\bs{Y}(\bs{\tau}+\bs{t}_i))   \, \Big\vert \, \mathcal{F}_{\bs{\tau}}\right] = \mathbb{E}_{\bs{Y}(\bs{\tau}) }\left[\prod_{i=1}^k f_i(\bs{Y}(\bs{t}_i))  \right].  
\end{equation}

\end{definition}


\begin{remark}
By a standard monotone class argument, the above definition is equivalent to: 
For any $ \bs y \in \mathbb{R}_{+}^{N}$ and positive $\mathcal{F}$-measurable random variable $V:\Omega \to \mathbb R_+$,
\[
\mathbb{E}_{\bs y} [\bs \theta_{\bs\tau} V | \mathcal{F}_{\bs\tau}] = \mathbb{E}_{\bs Y(\bs\tau)} [ V ] 
\]
See \cite[Proposition 5.29]{ccinlar2011probability} for a reference.
\end{remark}

The following lemma provides a sufficient condition for the strong Markov property of multi-parameter processes. 

\begin{lemma}\label{lem:SMP}
If $Y_i,\,  i = 1, \cdots, N$ are independent strong Markov processes,  then strong Markov property holds at any stopping point $\bs{\tau} < \infty $ a.s.  for the multi-parameter process $\bs{Y} = (Y_1, \cdots, Y_N)$. 
\end{lemma}

\begin{proof} By a monotone class argument, it suffices to prove (\ref{eq:strongmarkov}) for $f_i$ of the form \[
f_i(\bs y) = \prod_{j = 1}^{N} f_{ij}(y_j), \quad f_{ij} \in  \mathcal{C}_b(\mathbb{R}), \, \,  j = 1,2, \cdots, N. 
\]

First we assume that there exists a countable sequence $\{\bs{s}_n\}_{n\in \mathbb{N}}$ in $\mathbb{R}^N_+$ such that \[\mathbb{P}(\bs{\tau} \in \{\bs{s}_n\}_{n\in \mathbb{N}}) = 1 .\]  Thanks to the property of conditional probability, \[
\mathbb{E} \left[ \prod_{i=1}^k \prod_{j=1}^N f_{ij}(Y_j(\tau_j+t_{ij}))   \, \middle | \, \mathcal{F}_{\bs{\tau}}\right] =\sum_{n = 1}^{\infty}\mathbbm{1}_{\{\bs\tau = \bs s_n\}}\mathbb{E} \left[ \prod_{i=1}^k \prod_{j=1}^N f_{ij}(Y_j(s_{nj}+t_{ij}))   \, \Big\vert \, \mathcal{F}_{\bs{s}_n}\right]   
\]
By the independence of $\{Y_i\}_{i \in \{1,2,\cdots, N\}}$ and Markov property of each $Y_i$, for each $n \in \mathbb{N}$ \begin{align*}
\mathbb{E} \left[ \prod_{i=1}^k \prod_{j=1}^N f_{ij}(Y_j(s_{nj}+t_{ij}))   \, \middle | \, \mathcal{F}_{\bs{s}_n}\right]   
&=
\prod_{j=1}^N \mathbb{E}_{Y_j(\bs s_{nj})}\left[ \prod_{i=1}^k  f_{ij}(Y_j(t_{ij}))   \right]  \\
&=
\mathbb{E}_{\bs Y({\bs s_n})}\left[ \prod_{i=1}^k \prod_{j=1}^N f_{ij}(Y_j(t_{ij}))   \right] .
\end{align*}
Thus  (\ref{eq:strongmarkov}) holds if $\bs\tau$ takes countable many values. 

Passing to the general case we approximate an arbitrary stopping point $\bs{\tau}$ by elementary stopping points,
\[
\bs{\tau} ^{(n)}= \sum_{\bs{s} \in \mathbb{N}^N}2^{-n}\bs{s} \mathbbm{1}_{(2^{-n}\bs{s}, 2^{-n}(\bs{s}+\bs 1)] }(\bs{\tau}),
\]
where $\bs s+\bs 1:=(s_1+1,\dotso,s_N+1)$. For each $n \in \mathbb{N}$, \eqref{eq:strongmarkov} holds at $\bs \tau^{(n)}$, $\bs \tau^{(n)}$ is a decreasing sequence of stopping points, and $\lim_{n\to \infty} \bs \tau^{(n)} = \bs \tau$.  Noticing that for each $j\in \{1,2,\cdots, N\}$, $Y_j$  is a Feller process with continuous paths, it is clear that $\lim_{n \to \infty} Y_j(\tau^{(n)}_j) =  Y_j({\tau_j})$. By a standard monotone class argument and the definition of Feller process, we get the continuity of $y \mapsto \mathbb E_y\left[ \prod_{i=1}^k f_i(Y_j(t_{ij}))  \right]$, which implies 
\begin{equation*}
\lim_{n \to \infty} \mathbb{E}_{Y_j(\tau^{(n)}_j) }\left[\prod_{i=1}^k f_i(Y_j(t_{ij}))  \right]  = 
\mathbb{E}_{Y_j(\tau_j) }\left[\prod_{i=1}^k f_i(Y_j(t_{ij}))  \right];
\end{equation*}
see \cite[Remark 5.35]{ccinlar2011probability} for a reference.
It gives rise to the equality 
\begin{equation*}
\lim_{n \to \infty} \mathbb{E}_{\bs{Y}(\bs{\tau^{(n)}}) }\left[\prod_{i=1}^k f_i(\bs{Y}(\bs{t}_i))  \right]  = 
\mathbb{E}_{\bs{Y}(\bs{\tau}) }\left[\prod_{i=1}^k f_i(\bs{Y}(\bs{t}_i))  \right] .
\end{equation*}

To conclude the result, it remains to show that 
\begin{align}\label{eq:limit}
    \lim\limits_{n \to \infty} \mathbb{E} \left[ \prod_{i=1}^k f_i(\bs{Y}(\bs{\tau^{(n)}}+\bs{t}_i))   \, \middle | \, \mathcal{F}_{\bs{\tau^{(n)}}}\right]= \mathbb{E} \left[ \prod_{i=1}^k f_i(\bs{Y}(\bs{\tau}+\bs{t}_i))   \, \middle | \, \mathcal{F}_{\bs{\tau}}\right].
\end{align}
To this end, we estimate the difference
\begin{equation*}
\left| \mathbb{E} \left[ \prod_{i=1}^k f_i(\bs{Y}(\bs{\tau^{(n)}}+\bs{t}_i))   \, \middle | \, \mathcal{F}_{\bs{\tau^{(n)}}}\right] - \mathbb{E} \left[ \prod_{i=1}^k f_i(\bs{Y}(\bs{\tau}+\bs{t}_i))   \, \middle | \, \mathcal{F}_{\bs{\tau}}\right]\right| 
\leq I+ II,
\end{equation*}
where \begin{align*}
I &=  \left| \mathbb{E} \left[ \prod_{i=1}^k (f_i(\bs{Y}(\bs{\tau^{(n)}}+\bs{t}_i))   - f_i(\bs{Y}(\bs{\tau}+\bs{t}_i)) ) \, \middle | \, \mathcal{F}_{\bs{\tau^{(n)}}}\right] \right|, \\
II & = \left| \mathbb{E} \left[ \prod_{i=1}^k f_i(\bs{Y}(\bs{\tau}+\bs{t}_i))   \, \middle | \, \mathcal{F}_{\bs{\tau^{(n)}}}\right] - \mathbb{E} \left[ \prod_{i=1}^k f_i(\bs{Y}(\bs{\tau}+\bs{t}_i))   \, \middle | \, \mathcal{F}_{\bs{\tau}}\right]\right|.
\end{align*}

Let us denote $$Z_n :=  \mathbb{E} \left[ \prod_{i=1}^k (f_i(\bs{Y}(\bs{\tau^{(n)}}+\bs{t}_i))   - f_i(\bs{Y}(\bs{\tau}+\bs{t}_i)) ) \, \middle | \, \mathcal{F}_{\bs{\tau^{(n)}}}\right].$$ Then by Jensen's inequality and the tower property of conditional expectation $$\mathbb{E}[|Z_n|] \leq \mathbb{E} \left[ \left| \prod_{i=1}^k (f_i(\bs{Y}(\bs{\tau^{(n)}}+\bs{t}_i))   - f_i(\bs{Y}(\bs{\tau}+\bs{t}_i)) )\right| \right],$$ which converges to $0$ by the dominated convergence theorem. The $L^1$ convergence of $\lim_{n \to \infty} \mathbb{E}[|Z_n|] = 0$  implies that, there exists a subsequence $\{n_j\}_{j \in \mathbb{N}}$ such that $\lim_{n_j \to \infty} Z_{n_j} = 0$ almost surely. Therefore $I$ converges to $0$ as $n_j \to \infty$. For the second term $II$, we apply the backward martingale convergence theorem
to obtain that $II \to 0$ as $  n\to \infty$. Hence,  we prove \eqref{eq:limit}, and \eqref{eq:strongmarkov} holds for all stopping point $\bs \tau$.
\end{proof}

Let us prove that the process $ X(\cdot)$ constructed above is a Markov process on the filtered probability space $\left(\Omega, \mathcal{F}, (\mathcal{F}_{\bs T(t)}), \mathbb{P}\right)$. 

\begin{prop}
The process $ X$ defined in \eqref{eq:diffusion}  is Markovian with the right-continuous filtration $(\mathcal{F}_{\bs T(t)})$. 
\end{prop}

\begin{proof}

Since $\{\bs Y(\bs t)\}_{\bs t \in \mathbb{R}^{N}_+}$ is a continuous multi-parameter process and for any $s\in \mathbb{R}^{N}_+$, $\bs T(s)$ is a stopping time w.r.t. filtration $(\mathcal{F}_{\bs t})$, it can be easily seen that $ X(s)=\bs Y(\bs T(s))$ is $\mathcal{F}_{\bs T(s)}$-measurable. Then we verify that the time changed filtration $(\mathcal{F}_{\bs T(s)})$ is right-continuous. To this end, fix $s_0 \geq 0$ and let $A \in \bigcap_{s >s_0} \mathcal{F}_{\bs T(s)}$. By the definition of stopping point, for each $s >s_0$, and $\bs t \in \R_+^N$
\[ A \cap  \{\bs T(s) \leq \bs t\} \in \mathcal{F}_{\bs t}=\mathcal{F}_{\bs t +},\]
where the last equality is due to the right continuity of filtration $(\mathcal{F}_{\bs t})$ from Lemma~\ref{lem:rightcontinuity}. By standard arguments this condition is equivalent to 
\[ A \cap \{\bs T(s) < \bs t \} \in \mathcal{F}_{\bs t}, \quad \forall \, \bs t \in \mathbb{R}_+^N.\]
Thanks to the continuity of stopping points $s \mapsto \bs T(s)$,
\[A \cap \{\bs T(s_0) <\bs t \}=\bigcap_{s>s_0} A \cap \{ \bs T(s) < \bs t\} \in \mathcal{F}_{\bs t}, \quad \forall \bs t \in \mathbb{R}_+^N, \]
which concludes that $A \in \mathcal{F}_{\bs T(s_0)}$.

Next, to verify the Markov property of $X$, it is sufficient to show that 
\[\E_x [ f(X(t+h)) \, | \, \mathcal{F}_{\bs T(t)}]= \E_{X(t)}[ f(X(h))]\]
for any $f \in C_0(\Gamma)$. By Tietze extension theorem, $f$ can be viewed as a bounded continuous function on $\R_+^N$, and therefore $f(\bs Y(\bs T(t)))$ is well-defined. According to the definition of $X$, it is equivalent to show that 
\begin{align}\label{eq:MP}
\E_{\bs y}[ f(\bs Y( \bs T(t+h))) \, | \, \mathcal{F}_{\bs T(t)}]=\E_{ \bs Y( \bs T(t))} [ f(\bs Y(\bs T(h)))] , 
\end{align}
where $\bs y \in \mathbb{R}_+^N$ is the representation of $x \in \Gamma$; see \eqref{eq:representation}. 

According to Proposition~\ref{prop:T} \emph{(iii)}
\begin{align*}
    f(\bs Y( \bs T(t+h)))= f(\bs Y(\bs T(t) +\bs \theta_{\bs T(t)}\bs T(h))).
\end{align*}
Remember that $\bs Y$ is defined as the canonical process on $\Omega=\mathcal{C}([0,\infty);\mathbb R)^N$, and therefore 
\[f(\bs Y(\bs T(t) +\bs \theta_{\bs T(t)}\bs T(h))(\omega))=f(\bs Y( \bs T(h))(\bs\theta_{\bs T(t)}(\omega)))=\bs\theta_{\bs T(t)}  f(\bs Y( \bs T(h)))(\omega), \]
where the last equality is provided by the definition of the shift operator $\bs \theta$. Taking conditional expectation w.r.t. $\mathcal{F}_{\bs T(t)}$, the two equations above yield
\begin{align*}
\E_{\bs y}[ f(\bs Y( \bs T(t+h))) \, | \, \mathcal{F}_{\bs T(t)}] =  \E_{\bs y}[ \theta_{\bs T(t)}  f(\bs Y( \bs T(h))) \, | \, \mathcal{F}_{\bs T(t)}].
\end{align*}
We completes the proof of \eqref{eq:MP} by invoking the strong Markov property of $\bs Y$ from Lemma~\ref{lem:SMP}.

\end{proof}

Define the semigroup $(P_t)$ of the Markov process $X$ as the following.
\begin{align*}
P_t f (x): = \E_x[ f(X_t)], \quad \forall f \in C_0(\Gamma), \, x \in \Gamma. 
\end{align*}
Next we show that $X$ constructed above is a diffusion process on graph $\Gamma$ which is characterized by the infinitesimal generator ${A}$ and the corresponding domain $\mathcal{D}(A)$ as in \eqref{eq:infinitesimal}.  
\begin{prop}\label{prop:verification}
 For any $ x \in e_i$, $i=1,\dotso, N$, it holds that for all $f \in \mathcal{D}(A)$ $$\lim\limits_{t \to 0} \frac{P_tf(x)-f(x)}{t}=\mathcal{L}_if(x).$$
\end{prop}
\begin{proof}
 It suffices to prove the result for $x=v$. Let us define $|X(t)|:=d(X(t),v)$, where $d$ is the tree metric on graph $\Gamma$. 
 Define a sequence of downcrossing times iteratively for any fixed $\delta >0$
\begin{align*}
&\tau_0^{\delta}=0, \\
& \theta_1^{\delta}=\inf\{ t >  \tau_0^{\delta}: \, |X(t)| =\delta \}, \\
&\tau_1^{\delta}= \inf\{ t >\theta_1^{\delta}: \, |X(t)|=0\}, \\
& \, \, \vdots \\
& \theta_n^{\delta}=\inf\{ t >  \tau_{n-1}^{\delta}: \, |X(t)| =\delta \}, \\
&\tau_n^{\delta}= \inf\{ t >\theta_n^{\delta}: \, |X(t)|=0\}. 
\end{align*}
Then we split $P_t f(v)- f(v)$ as a telescopic sum
\begin{align*}
\E_{v}[f(X_t)-f(v)]=& \sum_{n\geq 0} \E_{v}[ f(X_{t \wedge \theta^{\delta}_{n+1}})-f(X_{t \wedge \tau^{\delta}_{n}})]  \\
&+\sum_{n\geq 1} \E_{v}[ f(X_{t \wedge \tau^{\delta}_{n}})-f(X_{t \wedge \theta^{\delta}_{n}})] =(I)^{\delta}+(II)^{\delta}. 
\end{align*}
Heuristically, $(I)^{\delta}$ is an approximation of dynamics of $f(X(t))$ around $v$, which can be estimated using local times. The term $(II)^{\delta}$ is the dynamic of $f(X(t))$ when $X(t)$ is away from $ v$, and we can compute it using It\^{o}'s formula. 

Note that one can characterize local times by downcrossings, and hence 
\begin{align*}
\lim\limits_{\delta \to 0} (I)^{\delta}&=\lim\limits_{\delta \to 0}    \sum_{i=1}^N D_i f (v) \left(\delta \sum_{n \geq 1} \mathbbm{1}_{ \{\tau_n^{\delta}<t,  X(\theta_n^{\delta}) \in e_i \}}\right) =  \sum_{i=1}^N D_i f(v) \, L^{i}(T_i(t)).
\end{align*}
According to our construction of $\{T_i\}_{i=1}^N$, and the definition of $\mathcal{D}(A)$, we obtain that 
\begin{align*}
\sum_{i=1}^N D_i f(v)L^{i}(T_i(t))=\sum_{i=1}^N \alpha_i D_i f(v) \frac{L^{i}(T_i(t))}{\alpha_i}=0,
\end{align*}
and hence that $\lim\limits_{\delta \to 0} (I)^{\delta}=0$. 

Let us now estimate $(II)^{\delta}$. Denote by $i(t)$ the index of edge which $X(t)$ belongs to, i.e., $i(t)=j$ iff $X(t) \in e_j$. As $\lim\limits_{\delta \to 0} \sum_{n \geq 0} ( \theta^{\delta}_{n+1}-\tau^{\delta}_n )=0$, it can be easily verified that 
\begin{align*}
\lim\limits_{\delta \to 0} (II)^{\delta}= \E_{v} \left[ \int_0^t \mathcal{L}_{i(s)}f(X(s)) \, ds\right].
\end{align*}
Now it is clear that 
\begin{align*}
\lim\limits_{t \to 0} \frac{P_tf(v)-f(v)}{t}= \lim\limits_{t \to 0} \frac{\E _{v}\left[ \int_0^t \mathcal{L}_{i(s)}f(X(s)) \, ds\right]}{t}= \mathcal{L}_if(v), \ \ \text{ for all $i=1,\dotso, N$,}
\end{align*}
where the last the equality is due to our definition of $\mathcal{D}(A)$ in \eqref{eq:infinitesimal}.
\end{proof}

\begin{proof}[Proof of Theorem~\ref{thm}] 
It was shown in \cite{MR1245308} that $(A,\mathcal{D}(A))$ is indeed an infinitesimal generator of a Feller semigroup. By the uniqueness part of Hille-Yosida theorem; see e.g. Lemma \cite[Lemma 17.5]{MR4226142} and Proposition~\ref{prop:verification}, this Feller semigroup must coincide with $(P_t)$ generated by $X(\cdot)$. Furthermore since $t \mapsto \mathcal{F}_{\bs T(t)}$ is right-continuous, it implies that $X(\cdot)$ is a strong Markov process with filtration $(\mathcal{F}_{\bs T(t)})$, which concludes the result. 
\end{proof}

Let us give some interpretations of positive constants $\alpha_i$, $L(t)$, and the time change $T(t)$.  It is intuitive to view $\alpha_i$ as the probability of choosing edge $e_i$ when the process is at vertex $v$. We state this observation in Remark \ref{directionProb}, the proof of which can be found in \cite[Corollary 2.4]{MR1743769}. 

\begin{remark} \label{directionProb}
	For each $i =1,\dotso, N$, we have that $\lim\limits_{\delta \to 0}\P_{v}[X(\theta^{\delta}) \in e_i]=\alpha_i$, where $\theta^{\delta}:=\inf\{ t > 0: \, |X(t)| =\delta \}$. 
\end{remark}

For process $X$, one can still define its local time as limits of downcrossing times as in the case of semimartingales. Adopting the notation from Proposition~\ref{prop:verification}, its local time at the vertex $v$ is given by 
\begin{align}\label{eq:localtimesveri}
    \lim\limits_{\delta \to 0} \delta \sum_{n \geq 0} \mathbbm{1}_{\{\tau^{\delta}_n<t\}}=\sum_{k=1}^N L^i(T_i(t))=\sum_{k=1}^NL^i(h_i(L(t))=\sum_{k=1}^N\alpha_iL(t)=L(t),
\end{align}
which justifies our notation in \eqref{eq:Xlocaltime}.

Let us turn to the time change $T(t)$. It is straightforward that for each $i$,  $T_i(t)$ is the occupation time of $X$ on edge $e_i$, i.e., \[
T_i(t) = \int_0^t \mathbf{1}_{\{X(s) \in e_i\}}(s) \, ds.
\]In Corollary \ref{ergodic}, we investigate the long-term average occupation time for a special case when all edges of $\Gamma$ have finite lengths. This result aligns with \cite[Remark 3.1]{MR4003554}, which we are able to obtain using a simpler approach thanks to our construction of time change $\mathbf{T}$.

\begin{cor}\label{ergodic}
	Suppose each edge $e_i$ is homeomorphic to $[0,l_i]$ for some positive real number $l_i$, then for each $i \in \{1,2,\cdots, N\}$, we have that , \[
	\lim_{t\to \infty}\frac{T_i(t)}{t} = \frac{\alpha_i m_i([0, l_i])}{\sum_{j = 1}^N \alpha_j m_j([0, l_i])} \ \ a.s.,
	\]
	where $m_i$ is the speed measure of process $Y_i$ on the interval $[0, l_i]$, i.e., \[
	m_i([0, l_i]) =\int_0^{l_i }\frac{2}{\sigma_i^2(x)}\exp\left({\int_0^x \frac{2b_i(y)}{\sigma_i^2(y) }dy}\right)  dx.
	\]
\end{cor}
\begin{proof}
	By the construction of $T_i$,  we have \[
	h_i(l) \le T_i(t) \le h_i(l)+h(l+) - h(l), \, \forall l\ge 0,  t \in [h(l), h(l+)].
	\]
	Thus\[
	\frac{h_i(l)}{h(l)} \frac{h(l)}{h(l+)} \le \frac{T_i(t)}{t} \le \frac{h_i(l)+h(l+) - h(l)}{h(l)} ,  \, \forall l\ge 0,  t \in [h(l), h(l+)].
	\]
	Notice that $h$ and all $h_j$ are strictly increasing functions of $l$, thus \[
	\lim_{l \to \infty}\frac{h_i(l)}{h(l)} \frac{h(l)}{h(l+)} \le \lim_{t\to \infty}\frac{T_i(t)}{t} \le \lim_{l \to \infty }\frac{h_i(l)+h(l+) - h(l)}{h(l)} .
	\]
	The left-continuous inverse of $L^i(\cdot)/\alpha_i$, $h_i (l)$,  can be rewritten as $h_i (l)= \rho_i(\alpha_i l)$ where $\rho_i$ is the left-continuous inverse of $L^i(\cdot)$. Thus  $h_i (l)$ has independent stationary increment and $\mathbb{E}[h_i(l)] = \mathbb{E}[\rho_i(\alpha_i l)] = \alpha_i l m_i([0, l_i]) $, see \cite[Chapter II, Section 2.14]{borodin2002handbook}.  By the strong law of large numbers we have that  \[
	\lim_{l \to \infty }\frac{h_i(l)}{l} = \lim_{l \to \infty }\frac{h_i(l+)}{l} = \mathbb{E}[h_i(1)] = \alpha_i m_i([0, l_i]) \ \  a.s.
	\]
	Using $h(\cdot) = \sum_{j = 1}^N h_j(\cdot)$, we obtain that \[
	\lim_{t \to \infty} \frac{T_i(t)}{t} =  \frac{\alpha_i m_i([0, l_i])}{\sum_{j = 1}^N \alpha_j m_j([0, l_i])} \ \ a.s.
	\]
\end{proof}

Let us also mention that our construction can be generalized to diffusions on metric graphs, which are singular processes behaving as Walsh diffusions in neighborhood of each vertex; see \cite{MR1245308} for the precise definition.  
Existence of such diffusions has been studied using Feller's semigroup theory.  An analog of It\^{o}'s formula, and a large derivation principle have also been established in \cite{MR1743769}. In \cite{doi:10.1063/1.4714661}, Brownian motions on a metric graph are defined rigorously and constructed pathwise. By merging two graphs as described in  \cite[Section 4]{doi:10.1063/1.4714661},  we can glue Walsh diffusions constructed in our paper and obtain any desired diffusion on a general metric graph. 

On the other hand, our construction can be generalized directly to diffusions on a metric graph under the assumption that there is no loop in the graph. The key idea is to induct on the number of internal vertices of graph. To elaborate the idea, let us consider a graph with two internal vertices $v$ and $u$ as in Figure~\ref{fig1}. We can view $\Gamma_v$, the star shaped subgraph with vertex $v$, as an edge of the vertex $u$, and carry out the same construction for Walsh diffusions.

For any diffusion process on a graph  $\Gamma$ with $n$ internal vertices and any given initial condition $X(0)$, we can divide the graph $\Gamma$ into some subgraphs  $\Gamma_i, i = 1, \cdots, M$ with internal vertices strictly less than $n$ and the initial condition of diffusion $\tilde Y_i$ on each subgraph $\Gamma_i$ is uniquely determined by the shortest path to $X(0)$ (Since  there are no loops in the graph, the shortest path is unique). We obtain the desired diffusion on $\Gamma$ via a time change of the multi-parameter process whose components are the diffusion processes on these subgraphs, i.e. $X(t) = \left( \tilde Y_1(\tilde T_i(t)), \cdots, \tilde Y_M(\tilde T_M(t))\right)$. By induction, $X(t) = \left( Y_1(T_1(t)), \cdots, Y_N(T_N(t))\right)$.

\begin{figure}[H]
    \centering
    \includegraphics[width=0.5\textwidth]{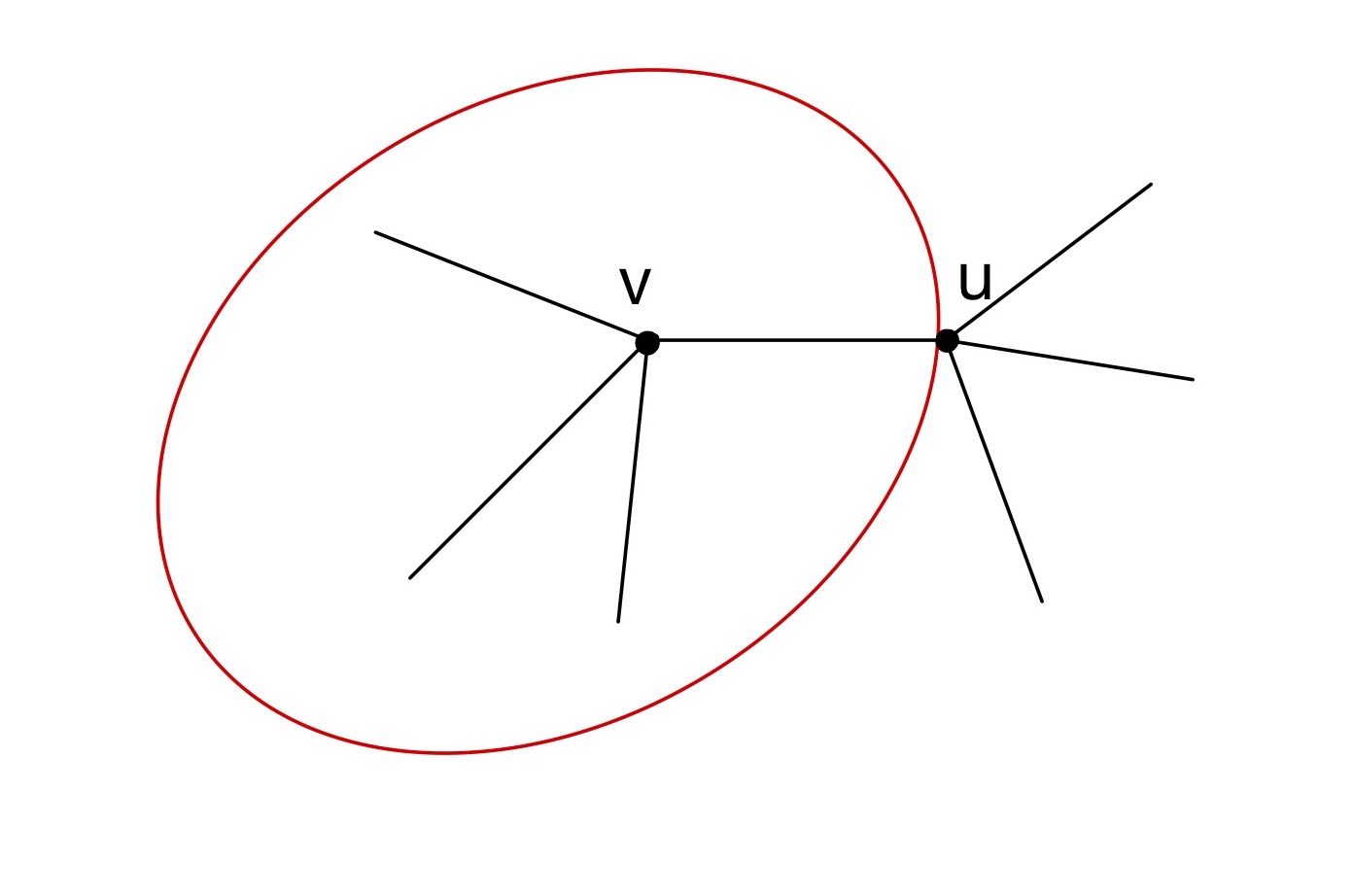}
    \caption{A graph with two internal vertices}
    \label{fig1}
\end{figure}

\bibliographystyle{siam}
\bibliography{ref}
\end{document}